\documentclass[12pt]{amsart}
\usepackage{amssymb,euscript}
\newtheorem{theorem}{Theorem}[section]
\newtheorem{lemma}[theorem]{Lemma}
\newtheorem{proposition}[theorem]{Proposition}
\newtheorem{corollary}[theorem]{Corollary}
\theoremstyle{definition}

\theoremstyle{remark}
\newtheorem*{remark}{Remark}
\newtheorem*{notation}{Notation}
\renewcommand{\hat}{\widehat}
\renewcommand{\tilde}{\widetilde}

\begin{document}



\title{Nilpotent primitive linear groups over finite fields}



\author{A. S. Detinko}
\address{ Department of Computer Science,
School of Computing and Engineering,
University of Huddersfield, Queensgate,
Huddersfield
HD1 3DH, UK}
\email{detinko.alla@gmail.com}
\author{D. L. Flannery}
\address{Department of Mathematics, University
 of Galway, Galway H91 TK33, Ireland}
\email{dane.flannery@universityofgalway.ie}

\maketitle



\section{Introduction}

 In this paper we investigate the structure
 of groups as in the title. Our work builds on work of
 several other authors, namely Konyukh \cite{Konyuh},
 Leedham-Green and Plesken \cite{LGP}, 
 and  Zalesskii \cite{Zalesskii}, who  have described 
 the abstract isomorphism types of the groups. We obtain 
 more detailed descriptions, in particular
 explaining how group structure depends on the existence 
 of an abelian primitive 
 subgroup. Additionally we show that isomorphism
 type of a group completely determines its 
 conjugacy class in the relevant general linear group.

 A  brief outline of the paper now follows.
 In Section~\ref{ablin} we review standard material on
 abelian (cyclic) irreducible linear groups.
 In Section~\ref{fund} fundamental structural results are 
 given. In Section~\ref{deg2} nilpotent primitive linear 
 groups of degree 2 are classified up to conjugacy, 
 and then groups of degree greater than 2 are treated 
 thoroughly in Section~\ref{greater2}. 
 The final Section~\ref{summary} summarises our results. 

 Throughout, $\mathbb{F}$ is a finite field of size $q$ and 
 characteristic $p$, and $G \leq \mathrm{GL}(n,\mathbb{F})$,
 $n>1$. The natural (right) $\mathbb{F}G$-module 
 of dimension $n$ is denoted $V$. Whenever we
 refer to  ``primitive'' or ``imprimitive'' 
 linear groups, we are implicitly assuming 
 them to be irreducible.
%

\section{Abelian linear groups}
\label{ablin}

 Suppose $A$ is an abelian irreducible subgroup of 
 $\mathrm{\mathrm{GL}}(n, \mathbb{F})$, and   
 denote by $\Delta$ the enveloping algebra 
 $\langle A  \rangle _\mathbb{F}$ of $A$.
\begin{proposition}
\label{singdef}
 $\Delta$ is a field, $|\Delta : \mathbb{F}1_n| =n$,
 so that $A$ is cyclic of order dividing $q^n-1$. 
 Hence  $\Delta^{\! \times}$ is a maximal abelian  
 subgroup of $\mathrm{GL}(n, \mathbb{F})$.
\end{proposition}
\begin{proof}
 The first two claims are covered by 
 \cite[Theorem 4, p.99]{Suprunenko}. The rest is then clear.
\end{proof}
\begin{corollary}
 A subgroup $C$ of $\Delta^{\! \times}$ is irreducible if 
 and only if
 $|\langle C \rangle_{\mathbb{F}} : \mathbb{F}1_n| =n$.
\end{corollary} 
%

 There always exist elements of ${\rm GL}(n,\mathbb{F})$ of 
 order $q^n-1$, and if $g$ is any such element then
 $\langle g  \rangle$ is irreducible.
 A cyclic subgroup of $\mathrm{GL}(n,\mathbb{F})$ of order
 $q^n-1$ is called a Singer cycle. The next few results
 about Singer cycles are well-known.
\begin{proposition}
\label{charirredab}
 A subgroup $C$ of $\Delta^{\! \times}$ is irreducible if 
 and only if $|C|$ does not divide $q^m-1$ for any 
 proper divisor $m$ of $n$.  
\end{proposition}
\begin{proposition}
\label{normcentr}
 Let $C$ be an irreducible subgroup of $\Delta^{\! \times}$.
 The $\mathrm{GL}(n,\mathbb{F})$-centraliser of 
 $C$ is $\Delta^{\! \times}$, and the 
 $\mathrm{GL}(n,\mathbb{F})$-normaliser of 
 $C$ is the semidirect product of   
 $\Delta^{\! \times}$ with a cyclic group of order
 $n$.
\end{proposition}
\begin{proposition}
\label{conjab}
 Abelian irreducible subgroups 
 of $\mathrm{GL}(n,\mathbb{F})$
 of the same order are conjugate.
\end{proposition}

 Now suppose $A$ is imprimitive, with imprimitivity
 system $\mathcal{S}$ $=$ $\{ V_1, \ldots \! \, ,V_k\}$, 
 $k > 1$. That is, the 
 components $V_i$ are subspaces of $V$,  
 $V = V_1 \oplus \cdots \oplus V_k$, 
 and $A$ permutes the $V_i$ by its natural action.
 In fact $A$ acts transitively because it is irreducible. 
 Hence $k$ divides $|A|$ and $n$, and 
 $k$ is not divisible by $p$.
 
 Let $\phi$ be the homomorphism  $A \rightarrow \mathrm{Sym}(k)$
 arising from the
 action of $A$ on $\mathcal{S}$. An abelian transitive permutation 
 group is regular, so $\phi(A)$ is cyclic of order $k$.     
\begin{lemma}
\label{sizeprim}
 For each integer $r>1$ dividing $k$, $A$  
 has a system of imprimitivity consisting of $r$
 components.
\end{lemma}
\begin{proof}
 Set $k/r=l$ and $\phi(A) = \langle h\rangle$.  
 Then $\mathcal{S}$ is a disjoint union of $r$ 
 $\langle h^r \rangle$-orbits 
 $\{ V_{i_1} , \ldots \, \! , V_{i_l}\}$ of length $l$, 
 and $\{ V_{i_1} \oplus \cdots \oplus V_{i_l} 
 \mid 1 \leq i \leq r \}$ 
 is an $A$-system of imprimitivity.
\end{proof}

 By Lemma~\ref{sizeprim} we can consider that 
 $A$ has an imprimitivity system of prime size.
\begin{proposition}
\label{likesim}
 Let $k$ be prime, and denote $\ker \phi$ by $A'$.
\begin{itemize}
\item[{\rm (i)}]  Each $V_i$   
 is a faithful irreducible $A'$-module.
\item[{\rm (ii)}]  
 $|A|$ divides $k(q^{n/k}-1)$.
\end{itemize}
\end{proposition}
\begin{proof}
 The stabiliser $A_i$ $=$
 $\{ a\in A \mid V_i a =V_i \}$ acts irreducibly
 on $V_i$, $1\leq i \leq k$  
 (by \mbox{e.g.} \cite[Theorem 4.2B, p.68]{Dixon}).
 We have $A'= \cap_i A_i \leq A_i$, and because $|A : A'|$ is prime
 and $A$ is irreducible, $A'=A_i$ for all $i$. 
 If $a \in A$ fixes $V_i$ elementwise then by transitivity
 of $A$ on $\mathcal{S}$, $a$ fixes every element of
 $V$. That is, $V_i$ is a faithful $A_i$-module. Hence (ii) 
 follows by Proposition~\ref{singdef}.   
\end{proof}

 The converse of Proposition~\ref{likesim} (ii) is also true
 (see \cite[2.6]{Sim}).
\begin{proposition}
\label{likersim}
 A subgroup $C$ of  $\Delta^{\! \times}$  is 
 imprimitive if and only if $|C|$ divides 
 $k(q^{n/k}-1)$ for some prime $k$ dividing  $n$.
\end{proposition}

\section{Fundamental structural results}
\label{fund}

 We frequently use the following two lemmas, 
 whose straightforward proofs are omitted.
\begin{lemma}
\label{abnorsgp}
 A finite abelian normal subgroup of a primitive
 linear group 
 {\em (}over an arbitrary field\hspace*{.7mm}{\em )} 
 is cyclic.
\end{lemma}

\begin{lemma}
\label{ind2prim} 
 Any subgroup of index $2$ of a primitive linear group 
 is irreducible.
\end{lemma}

\begin{notation}
 If $H$ is nilpotent then we denote its Sylow $2$-subgroup 
 by $H_2$.
\end{notation}

 The following theorem 
 is vital.
\begin{theorem}
\label{konzal}
 Suppose $G$ is nilpotent primitive.  
\begin{itemize}
\item[{\rm (i)}] Every odd order subgroup of $G$ is cyclic.
\item[{\rm (ii)}] If $G_2$ is nonabelian then 
 $q\equiv 3$ mod $4$, and $G_2$ has maximal nilpotency 
 class: either $G_2 \cong Q_8$,
 or $G_2$ has order at least $16$ 
 and is generalised quaternion, 
 dihedral, or semidihedral. 
\end{itemize}
\end{theorem}
\begin{proof}
 See \cite{Zalesskii}
 or \cite{Konyuh}. 
 Everything  derives from the fact that $[G,G]$ is cyclic.
\end{proof}
 
 Thus nilpotent primitive $G$ has the form $G_2 \times C$, 
 where the cyclic group $C$ is the direct product of the odd 
 order Sylow subgroups of $G$, and $G$ has a cyclic subgroup 
 of index 2. 

 A nilpotent primitive linear group over a finite field is 
 metacyclic (normalises a Singer cycle). Sim \cite{Sim} 
 classifies all metacyclic primitive linear groups of odd 
 prime power degree over finite fields. That case is not 
 relevant to our purposes, by the next proposition. 
\begin{proposition}
\label{oddfact}
 If $G$ is a  nonabelian nilpotent primitive subgroup
 of ${\rm GL}(n , \mathbb{F})$ then $n$ is even.
\end{proposition}
\begin{proof}
 By Theorem~\ref{konzal}, $q\equiv 3$ mod 4
 and $G$ has a cyclic normal subgroup $H$ 
 of index 2 and order divisible by 4. 
 By Lemma~\ref{ind2prim} and Proposition~\ref{singdef}, 
 $|H|$ divides $q^n-1$, whereas if $n$ is odd 
 then $q^n-1$ is not divisible by 4.
\end{proof}

\section{Degree 2}
\label{deg2}

\begin{proposition}
\label{absirredcase1}
 Suppose $G$ is nilpotent primitive. 
 Then $G$ is absolutely irreducible if and only 
 if $n=2$ and $G$ is nonabelian. 
\end{proposition}
\begin{proof} 
 Since every element of $G$ is semisimple  
 (\cite[Corollary 1, p.239]{Suprunenko}),
 $p$ does not divide $|G|$. 
 Thus $G$ is completely reducible over the algebraic 
 closure of $\mathbb{F}$, so 
 that if $n=2$ and $G$ is nonabelian then $G$ must be 
 absolutely  irreducible.
 Conversely, if $G$ is absolutely irreducible then it 
 is nonabelian, and isomorphic to an irreducible subgroup 
 of $\mathrm{GL}(n,\mathbb{C})$---see e.g. 
 \cite[Corollary 3.8, p.62]{Dixon}.
 By Theorem~\ref{konzal}, 
 $G$ has an abelian subgroup of index 2.  
 Therefore $n$ divides 2 by Ito's result 
 \cite[6.15, p.84]{Isaacs}.
\end{proof}

\begin{proposition}
\label{secchar}
 A nonabelian subgroup $G$ of
 $\mathrm{GL}(2,\mathbb{F})$ is nilpotent primitive 
 if and only if $q \equiv 3$ mod $4$ 
 and $G = G_2 \times C$, where $G_2$ is primitive
 and $C$ is a {\rm (}cyclic{\rm )} group of odd order 
 scalars.
\end{proposition}
\begin{proof}
 Suppose $G$ is nilpotent primitive, so $q \equiv 3$ mod 4.
 Then a Sylow 2-subgroup of the full monomial subgroup 
 of $\mathrm{GL}(2,\mathbb{F})$ is dihedral of order 8. 
 By Theorem~\ref{konzal} (ii), $G_2$ must be primitive, 
 because it is not monomial (the only sort of imprimitivity 
 in prime degree). Furthermore, $G_2$ is absolutely 
 irreducible by Proposition~\ref{absirredcase1}. Schur's 
 Lemma implies $C$ is scalar. 
\end{proof}
\begin{remark}
 We emphasise that the isomorphism types for $G_2$ 
 in Proposition~\ref{secchar} may be written down 
 from Theorem~\ref{konzal} (ii); cf. also \cite[II.2]{LGP}.
\end{remark}
\begin{remark}
 By \cite[III.4 (ii)]{LGP}, when $q\equiv 3$ mod 4,
 a nilpotent primitive subgroup of $\mathrm{GL}(n,\mathbb{F})$
 can have primitive Sylow 2-subgroup only if $n=2$.
\end{remark}

 Below we say a little more about 2-subgroups of 
 $\mathrm{GL}(2,\mathbb{F})$. Prior to that we present
 some results on conjugacy between nilpotent linear groups.
\begin{theorem}
\label{isoeqconj}
 Suppose $n\neq p$ is prime, and 
 $G$, $H$ are nonabelian $n$-subgroups of
 $\mathrm{GL}(n,\mathbb{F})$. Then $G$, $H$ are conjugate 
 in $\mathrm{GL}(n,\mathbb{F})$ if and only if 
 $G \cong H$.
\end{theorem}
\begin{proof}
 If $G \cong H$ then by \cite[Proposition 4.2]{Conlon},
 $G$ and $H$ are conjugate as subgroups of 
 $\mathrm{GL}(n,\overline{\mathbb{F}})$, where
 $\overline{\mathbb{F}}$ is the algebraic closure of
 $\mathbb{F}$. 
 Thus $G$ and $H$ are conjugate in $\mathrm{GL}(n,\mathbb{F})$ 
 by the Deuring-Noether Theorem 
 \cite[1.22, p.26]{HuppertBlackburn}.
\end{proof} 
\begin{corollary}
\label{conl}
 Two nilpotent primitive subgroups of 
 $\mathrm{GL}(2,\mathbb{F})$ are
 conjugate in $\mathrm{GL}(2,\mathbb{F})$ if and only
 if they are isomorphic.
\end{corollary}
\begin{proof}
 The result is clear for abelian groups by
 Proposition~\ref{conjab}; otherwise  
 we appeal to Proposition~\ref{secchar} and 
 Theorem~\ref{isoeqconj}.
\end{proof}

\begin{lemma}
\label{2gps}
 Suppose $q\equiv 3$ mod $4$ and let $\omega$ be 
 a generator of ${\rm GF}(q^2)^{\times}$. 
 If $2^t$ is the largest power of $2$ dividing 
 $q+1$ then the subgroup of  
 ${\rm GL}(2,\mathbb{F})$  generated by 
$$
x= \left(
\begin{array}{cc}
 0 & 1 \\
 -\omega^{q+1} & \omega + \omega^q 
\end{array} \right)^{(q^2-1)/2^{t+1}}
 , \ \ \ \ \ \ \ 
y= \left(
\begin{array}{cc}
 1 & 0 \\
 \omega + \omega^q & -1 
\end{array}\right) 
$$
 
\noindent is a Sylow $2$-subgroup, and is semidihedral of
 order $2^{t+2}$.
 There are dihedral and generalised quaternion subgroups
 of ${\rm GL}(2,\mathbb{F})$ of every order $2^s$, 
 $3 \leq s \leq t+1$, unique up to conjugacy 
 at every order. Up to conjugacy, the unique 
 {\rm (}nonabelian{\rm )} semidihedral $2$-group in  
 ${\rm GL}(2,\mathbb{F})$ is the Sylow $2$-subgroup 
 indicated above.
\end{lemma}
\begin{proof}
 We readily verify that $|x| = 2^{t+1}$, $|y|=2$, and 
 $T = \langle x \rangle \rtimes \langle  y \rangle$
 is semidihedral. 
 The uniqueness statement regarding the proper 
 subgroups of $T$ is an instance of 
 Corollary~\ref{conl}.  
\end{proof}
\begin{remark}
 It is not difficult to write down explicit generators for 
 all nonabelian 2-subgroups of $\mathrm{GL}(2,\mathbb{F})$
 from the Sylow 2-subgroup generators given in 
 Lemma~\ref{2gps}. 
\end{remark}
\begin{corollary}
\label{cclcount}
 Assuming the notation and conditions of 
 Lemma{\rm~\ref{2gps}}, the number of distinct 
 $\mathrm{GL}(2,\mathbb{F})$-conjugacy 
 classes of nonabelian nilpotent primitive  subgroups of 
 $\mathrm{GL}(2,\mathbb{F})$ is
 $r(t-1)$, where $r$ is the number of 
 positive integer divisors of $q-1$.
\end{corollary}
\begin{proof}
 If $G \leq \mathrm{GL}(2,\mathbb{F})$ is 
 nonabelian nilpotent primitive then $G=G_2 \times C$ 
 for some odd order subgroup $C$ of $\mathbb{F}^{\times}1_2$.
 There are precisely $r/2$ choices for $C$. Using 
 Theorem~\ref{konzal} and Lemma~\ref{2gps} we then count
 $2(r/2)+2(t-2)r/2$ possible isomorphism types for $G$
 in total, so the result follows from 
 Corollary~\ref{conl}.  
\end{proof}

\section{Degree greater than 2}
\label{greater2}
 
 Proposition~\ref{oddfact} places a restriction on the 
 degree of a  nilpotent primitive linear group over a finite
 field. 
 We give another restriction below. The idea is to
 connect  general degree to degree 2, by means of 
 Proposition~\ref{absirredcase1} and the following.
\begin{theorem}
\label{prewehr}
 Let $G$ be primitive.
 For some $m$ dividing $n$ 
 there is an isomorphism of $G$ onto an absolutely 
 irreducible primitive subgroup of
 $\mathrm{GL}(n/m,q^m)$. 
\end{theorem}
\begin{proof}
 See \cite[1.19, p.12]{Wehrfritz}.
\end{proof}
\begin{proposition}
\label{2odegree}
 There exist  nonabelian nilpotent primitive subgroups of 
 $\mathrm{GL}(n,\mathbb{F})$ only if $n = 2m$, 
 $m$ odd. 
\end{proposition}
\begin{proof}
 By Theorem~\ref{prewehr}, suppose 
 there exist absolutely irreducible nilpotent primitive
 subgroups of $\mathrm{GL}(n/m,q^m)$, $n/m>1$.
 Proposition~\ref{absirredcase1} dictates that 
 $n/m=2$, and $q$, $q^{m}$ are both congruent to 3 mod 4 
 by Theorem~\ref{konzal}. Thus $m$ is odd.
\end{proof}

 In the sequel, our discussion of the structure of a nonabelian 
 nilpotent primitive linear group $G$
 turns on whether or not $G$ has cyclic primitive subgroups. 
 Once a cyclic subgroup $A$ of $G$ is known to be 
 irreducible, a simple order criterion 
 (Proposition~\ref{likersim}) can be used to test
 primitivity of $A$. If $A$   
 is found to be imprimitive, then it must have 
 an imprimitivity system of size 2 
 (Theorem~\ref{evenimprim}).
\begin{lemma}
\label{nakeasy}
 Let $G$ be irreducible, with a proper subgroup $K$ 
 of index $k$. 
\begin{itemize}
\item[{\rm (i)}]  The dimension of a nonzero $K$-submodule 
  of $V$ is at least $n/k$.
 \item[{\rm (ii)}]  If $V$ has a $K$-submodule $U$ of
 dimension $n/k$ then 
 $G$ is imprimitive.
\end{itemize}
\end{lemma}
\begin{proof}
 This lemma (which is true for any field in place of $\mathbb{F}$)
 follows mainly from Nakayama Reciprocity; 
 cf. the proof of \cite[2.6]{Sim}.
\end{proof}

 For the rest of the section, $q\equiv 3$ mod 4
 and $n=2m$, $m>1$ odd.   
\begin{theorem}
\label{evenimprim}
 Let $G= G_2 \times C$ be nonabelian nilpotent primitive, 
 and let $A$ be a cyclic subgroup of $G$ of index $2$.
 If $A$ is imprimitive then every $A$-system of  
 imprimitivity consists of $2$ components. 
\end{theorem}
\begin{proof}
 Suppose $A$ has an imprimitivity system of odd prime
 size $t$. Let $D$ be the index $t$ subgroup of $C$.
 By Proposition~\ref{likesim} the dimension of an 
 irreducible $A_2D$-submodule of $V$ is  
 $n/t$. By Clifford's Theorem $G_2D$ is completely 
 reducible, with say $l$ irreducible parts of equal degree. 
 For the same reason a $G_2D$-submodule of $V$ 
 is completely reducible as an $A_2D$-module.
 Thus $n/l$ is divisible by $n/t$, so $l$ is 1 or $t$. 
 However $l\neq t$ by 
 Lemma~\ref{nakeasy} (ii). 
 Also $l \neq 1$, for if $G_2D$ were irreducible then 
 an $A_2D$-submodule of $V$ would have dimension at
 least $m$ by Lemma~\ref{nakeasy} (i).
 Thus the only prime dividing the size $s$ of an
 $A$-system of imprimitivity is 2 by Lemma~\ref{sizeprim},
 and as $s$ divides $2m$, we are done.
\end{proof}

 Until further notice, we fix the following notation and 
 hypotheses. Let $G = G_2 \times C$ where 
 $C = \langle c\rangle$ is odd order cyclic, 
 and $G_2$ is (nonabelian) generalised quaternion, 
 dihedral or semidihedral. 
 That is, $G_2 = \langle A_2, g\rangle$, 
 $|G_2 : A_2| = 2$, 
 and $A_2 = \langle a\rangle$, $|a| \geq 4$. 
 Suppose $A = \langle c, a\rangle$ is irreducible. 
 Then $g^2$ $\in$ $\langle -1_n \rangle \leq A_2$.
 We define a field automorphism $\delta$ of 
 $\Delta  = \langle A\rangle_\mathbb{F}$ by
 $\delta: x \mapsto g^{-1}xg$, 
 $x \in \Delta$. Let $\mathbb{K}$ be the $\delta$-invariant 
 subfield  $\{x \in \Delta: gx = xg\}$ of $\Delta$. 
 Note that $C \subseteq \mathbb{K}$. In addition
 $|\Delta : \mathbb{K}|$ $=$ $2$, so
 $|\mathbb{K} : \mathbb{F}1_n| = m$ and thus 
 $|\mathbb{K}|$ $\equiv$ 3 mod 4.
\begin{remark}
 $\mathbb{K}$ is the centraliser of $G$ in 
 $\mathrm{Mat}(n, \mathbb{F})$.
\end{remark}
\begin{lemma}
\label{delinv}
 $|\langle A_2\rangle_\mathbb{F} : \mathbb{F}1_n|=2$ 
 and $\langle A_2\rangle_\mathbb{F} \cap \mathbb{K}$ 
 $=$ $\mathbb{F}1_n$.
\end{lemma}
\begin{proof}
 Let $B$ be the subgroup of $A_2$ of order 4. Since 
 $B^2$ but not $B$ is scalar,  
 $|\langle B \rangle_\mathbb{F} : \mathbb{F}1_n| =2$.  
 An inductive argument then shows that 
 $|\langle A_2\rangle_\mathbb{F} : \mathbb{F}1_n| =2$.
 For the second claim we need only  
 observe that $\langle A_2 \rangle_\mathbb{F}$   
 contains an element of order 4 while $\mathbb{K}$ 
 does not.
\end{proof}

We denote a diagonal matrix as the ordered tuple of its main
diagonal entries.
\begin{lemma}
\label{prea2c}
If $\bar{A}\leq A$ then the irreducible $\bar{A}$-submodules
 of $V$ have dimension $|\langle \bar{A} \hspace*{.2mm} 
 \rangle_{\mathbb{F}} : \mathbb{F}1_n|$, 
and they are all isomorphic to each other.
\end{lemma}
\begin{proof}
 By Clifford's Theorem,  
 $\bar{A}$ is conjugate to  
 $\langle (x, \ldots \, \! , x) \rangle$ 
 for some irreducible subgroup $\langle x \rangle$
 of $\mathrm{GL}(r, \mathbb{F})$,
 $r=|\langle x\rangle_{\mathbb{F}}:\mathbb{F}1_{r}|$ 
 dividing $n$. 
 Clearly then $\langle \bar{A} \hspace*{.2mm} 
 \rangle_{\mathbb{F}}$ $\cong$ 
 $\langle  x\rangle_{\mathbb{F}}$.    
\end{proof}
\begin{lemma} 
\label{a2cirred}
\begin{itemize}
\item[{\rm (i)}] 
 An irreducible $A_2$-submodule of $V$ has
 dimension $2$.
\item[{\rm (ii)}] 
 $\mathbb{K} = \langle C\rangle_\mathbb{F}$, and an 
 irreducible $C$-submodule of $V$ has dimension 
 $m$.
\end{itemize}
\end{lemma}
\begin{proof}
 Lemmas~\ref{delinv} and \ref{prea2c} give (i).
 Since $\Delta$ is the  compositum of 
 $\langle C\rangle_{\mathbb{F}}$ and  
 $\langle A_2\rangle_{\mathbb{F}}$,  
 $\langle C\rangle_\mathbb{F}\subseteq \mathbb{K}$, and
 $\langle A_2\rangle_{\mathbb{F}} \cap \mathbb{K}$ 
 $=$ $\mathbb{F}1_n$, we get that
 $|\langle C \rangle_\mathbb{F}  : \mathbb{F}1_n| =m$
 i.e. $\langle C\rangle_\mathbb{F}=\mathbb{K}$,
 as required for (ii). 
\end{proof}
%
%
\begin{theorem}
\label{firstclassif}
 $G$ but not $A$ is primitive if and only if $A$ has an 
 imprimitivity system 
 of size $2$,  
 $C$ is conjugate to $\langle (x,x) \rangle$ 
 for some primitive subgroup 
$\langle x \rangle$ of 
 $\mathrm{GL}(m, \mathbb{F})$, and 
 $a^2 = g^2 = [a, g] = -1_n$.
\end{theorem}
\begin{proof}
 Suppose $G$ but not $A$ is primitive.
 By Theorem~\ref{evenimprim},
 $A$ has an imprimitivity system $\mathcal{S} = \{ V_1, V_2 \}$.
 Let $\phi$ be the permutation representation of $A$ in 
 $\mathrm{Sym}(2)$ arising from the action on $\mathcal{S}$, 
 and set $\ker \phi=A'$.
 Since $|C|$ is odd, necessarily $C \leq A'$, 
and then by
 Lemmas~\ref{prea2c} and \ref{a2cirred},
 $C = \langle (x,x) \rangle$ up to conjugacy,
 where $C_1 = \langle x \rangle \leq \mathrm{GL}(m, \mathbb{F})$ 
 is irreducible. Since $\langle C\rangle_\mathbb{F}$ $\leq$
 $\langle A'\rangle_\mathbb{F}$
 and $|\Delta : \langle C\rangle_\mathbb{F}| =2$, 
 either $\langle A' \rangle_\mathbb{F}$ $=$ 
 $\Delta$ or $\langle A'\rangle_\mathbb{F} = 
 \langle C\rangle_\mathbb{F}$. 
 If $\langle A' \rangle_\mathbb{F} = 
 \Delta$ then $A'$ is irreducible, 
 which is definitely not true.
 Thus $\langle A'\rangle_\mathbb{F} = 
 \langle C\rangle_\mathbb{F}$.    
 An element $h\neq 1$ of $A' \cap A_2$ has 2-power order;
 but $h$ $\in$  $\langle C\rangle_\mathbb{F}$, which
 does not contain an element of order 4, so $h = -1_n$. 
 Consequently 
 $A' = (A' \cap A_2)\times C = \langle -1_n, C \rangle$ and 
 $A'|_{V_i} = \langle -1_m, C_1 \rangle$. Since 
 $A'|_{V_i}$ is primitive (otherwise we get an $A$-system of
 imprimitivity of size greater than 2, violating 
 Theorem~\ref{evenimprim}), $C_1$ is primitive.
 From $|A: A'|=2$ and $|A'| = 2|C|$ we infer $|A_2| =4$
 and $|G_2|=8$. 
 By Theorem~\ref{konzal}, then, $G_2 \cong Q_8$.
 We have seen that $|a|=4$, and 
 $g \notin Z(G_2) = \langle -1_n \rangle$
 implies 
 $a^2=g^2= [a,g] = -1_n$.

 Now we prove the converse. Suppose 
 $C = \langle (x,x) \rangle$ where 
 $\langle x \rangle := C_1$ is a primitive subgroup of 
 $\mathrm{GL}(m, \mathbb{F})$, and $G_2 \cong Q_8$.
 Also suppose $\mathcal{T}$ $=$ $\{  W_1 , \ldots \, \! ,  W_k\}$
 is a $G$-system of imprimitivity, and let $\varphi
 : G \rightarrow \mathrm{Sym}(k)$ be the associated 
 permutation representation.
 Of course, $\mathcal{T}$ is simultaneously an $A$-system of 
 imprimitivity, so that
 if $k$ is divisible by an odd prime $t$
 then by Lemma~\ref{sizeprim}, $A$ has an imprimitivity 
 system of size $t$.
 By Proposition~\ref{likesim}, $|A| = 4|C_1|$ divides 
 $t(q^{m/t}-1)(q^{m/t}+1)$.  
 Some prime divisor $r$ of $|C_1|$ must divide $q^{m/t}+1$
 by Proposition~\ref{likersim}, and therefore $r$ divides $q^m+1$.
 However $|C_1|$ divides $q^m-1$, so $r$ divides $2q^m$, 
 a contradiction. 
 Hence $k=2$.
%
%
%
%
 
 Since
 $\varphi(A_2)$ is transitive, after replacing $G$ 
 with a suitable conjugate we have $a = (u, v) w$
 where $u, v$ $\in$ $\mathrm{GL}(m, \mathbb{F})$ and
$$
 w = 
 \left( \begin{array}{cc}
 0_m & 1_m \\
  1_m & 0_m 
 \end{array}
 \right) ,
$$
 with $C$ retaining its original form
 $\langle (x,x) \rangle$. The relation 
 $a^2 = -1_n$ yields $v$ $=$ $-u ^{-1}$.

 Recall Proposition~\ref{normcentr}.
 Since $a$ centralises $C$, $u$ centralises $C_1$,
 and therefore $u \in \langle C_1\rangle_\mathbb{F}$.
 Note that $d =  (1_m, -1_m)$ inverts $a$, so
  it normalises but does not centralise 
 $\Delta^{\! \times}$. Certainly $g$ normalises  
 $\Delta^{\! \times}$, so
  $g = db$ for some $b \in \Delta^{\! \times}$.
 
 Over $\langle C \rangle_\mathbb{F}$, $\Delta$ has basis
 $\{ 1_n, a\}$.
 Thus if 
%
 $g \in \ker \varphi$ then $b$ $=$ 
 $(y, y)$ for some
 $y\in \langle C_1\rangle_\mathbb{F}$. Since $g^2 = -1_n$
 it follows that $y^2 = -1_m$,
 which is impossible.

 Suppose $g \notin \ker \varphi$, so that
 $b$ $=$ $(u y ,-u^{-1} y)w$ for some 
 $y \in \langle C_1\rangle_\mathbb{F}$.
 Once again, $g^2 = -1_n$ implies 
 $\langle C_1\rangle_\mathbb{F}$ has an element
 of order 4. This final contradiction proves that $G$ is primitive. 
\end{proof}

 We next give an explicit description of 
 the groups in Theorem~\ref{firstclassif},
 which can
 be extracted from the proof of the theorem.
 An auxiliary fact is needed: for any prime 
 $p$, $-1$ is a sum of two squares mod $p$.
\begin{proposition}
 Let $C_1=\langle x \rangle$ be a cyclic primitive subgroup of 
 $\mathrm{GL}(m, \mathbb{F})$, and choose 
 $e , f \in  \langle C_1\rangle_\mathbb{F}$ such that
 $e^2+f^2=-1_m$. 
 Define $C = \langle (x, x)\rangle$ and 
 $G_2 = \langle a, g\rangle$, 
 where
$$a =\left( \begin{array}{rr}
0_m & 1_m \\
-1_m & 0_m 
\end{array} \right) , \ \ \ g = \left( \begin{array}{lrc}
e & f \\
f & -e 
\end{array} \right) .
$$
 Then $G = G_2 \times C$ is a nonabelian nilpotent 
 primitive subgroup of $\mathrm{GL}(n, \mathbb{F})$
 with an abelian imprimitive subgroup of index $2$.
 Moreover, any group of this kind with the 
 same order as $G$ is conjugate to $G$.
\end{proposition}

 Our last undertaking in this section is to extend 
 Corollary~\ref{conl}.
 We use the following version 
 of Theorem~\ref{prewehr}.
\begin{theorem}
\label{explwehr}
 Let $H$ be an irreducible but not absolutely 
 irreducible subgroup of ${\rm GL}(r,\mathbb{F})$, $r>1$. 
 For some $s>1$ dividing $r$ there is an 
 absolutely irreducible subgroup $\hat{H}\cong H$ of
 ${\rm GL}(r/s,q^s)$ such that $H$ is
 ${\rm GL}(r, q^s)$-conjugate to the 
 group $\tilde{H}$ of block diagonal matrices
$$
(h, h^{\sigma}, \ldots \, \! , 
h^{\sigma^{s-1}}) , \ \ \ \ h \in \hat{H} , 
$$
 where $\sigma$ is 
 a fixed generator 
 of $\mathrm{Gal}(\mathrm{GF}(q^s)/\mathbb{F})$, 
 and $h^{\sigma^j}$ denotes the matrix resulting from 
 entrywise action of $\sigma^j$ on $h$. 
\end{theorem}
\begin{proof}
 See \cite[9.21, p.154]{Isaacs} and the proof of that 
 result.
\end{proof}
\begin{theorem}
\label{iecagain}
 Nilpotent primitive subgroups $G$ and $H$ of 
 $\mathrm{GL}(n,\mathbb{F})$ 
 are conjugate in  $\mathrm{GL}(n,\mathbb{F})$
 if and only if they are isomorphic.
 \end{theorem}
\begin{proof}
 Suppose $G \cong H$.
 Since $G$ and $H$ are not
 absolutely irreducible ($n>2$) by Proposition~\ref{absirredcase1},
 there are absolutely irreducible 
 subgroups $\hat{G}$, $\hat{H}$ of $\mathrm{GL}(2,q^m)$, and 
 subgroups $\tilde{G}$, $\tilde{H}$ of $\mathrm{GL}(n,q^m)$
 constructed from  $\hat{G}$, $\hat{H}$ as in 
 Theorem~\ref{explwehr}, such that $G$ is conjugate to 
 $\tilde{G}$ and $H$ is conjugate to 
 $\tilde{H}$. The groups $\hat{G}$ and $\hat{H}$ are primitive, 
 so they are $\mathrm{GL}(2,q^m)$-conjugate by Corollary~\ref{conl}.
 Then it is easy to see that $\tilde{G}$ and $\tilde{H}$
 are $\mathrm{GL}(n,q^m)$-conjugate. 
 Hence $G$ and $H$ are $\mathrm{GL}(n,\mathbb{F})$-conjugate 
 by the Deuring-Noether Theorem.
\end{proof}

\section{Summary and concluding remarks}
\label{summary}

 Suppose $q\equiv 3$ mod 4. Let $G$ be nonabelian 
 nilpotent primitive, and let $C$ be the 
 direct product of the odd order Sylow subgroups of 
 $G$. Previously we have established that one of 
 the following situations must occur.
\begin{enumerate}
\item $n=2$, $G_2$ is absolutely irreducible, 
 and $C$ is scalar.
\item $n=2m$, $m>1$ odd, $G_2 \cong Q_8$, and
 $G$ has an imprimitive cyclic subgroup of index 2. 
\item $n=2m$, $m>1$ odd, and
 $G$ has a primitive cyclic subgroup of index 2. 
\end{enumerate}
 In all cases, $G_2$ is either $Q_8$,
  or a Sylow 2-subgroup of $\mathrm{GL}(n,\mathbb{F})$,
  or is generalised  quaternion or dihedral of order at least
 16. 
%

 Suppose $G \cong Q_8 \times C$
 and $n>2$. 
 Then $|C|$ divides $q^m-1$,
 so an index 2 subgroup of $G$ has order dividing 
 $2(q^m-1)$. Such a subgroup of 
 $G$ is imprimitive by Proposition~\ref{likersim}.
 Hence $|G_2| > 8$ in case 3, 
 %
 and for each possible value of $|C|$, i.e. 
 the order of an odd order cyclic primitive subgroup 
 of $\mathrm{GL}(m, \mathbb{F})$, 
 there is a single conjugacy class of primitive 
 subgroups of $\mathrm{GL}(n,\mathbb{F})$ isomorphic
 to $Q_8 \times C$. The $\mathrm{GL}(n,\mathbb{F})$-conjugacy 
 classes of the groups in case 3 may be counted with the
 aid of Theorem~\ref{iecagain} (cf. Corollary~\ref{cclcount}).  

\subsection*{Acknowledgment}
 The first author received support 
 from the Enterprise Ireland International Collaboration 
 Programme, and is also indebted
 to the Mathematics Department at NUI, Galway, where
 work on this paper was carried out.

\bibliographystyle{amsplain}

\end{document}